\documentclass[12pt,twoside]{article}
\usepackage{amssymb,amsmath,amsthm}
\usepackage[noadjust]{cite}

\usepackage[german,english]{babel}
\usepackage{geometry,xcolor}
\DeclareSymbolFont{bbold}{U}{bbold}{m}{n}
\DeclareSymbolFontAlphabet{\mathbbold}{bbold}
\newcommand{\ind}{\mathbbold{1}}

\newcommand{\C}{\mathbb{C}}

\newcommand{\N}{\mathbb{N}}
\newcommand{\R}{\mathbb{R}}

\newcommand{\Ls}{\mathcal{L}}

\renewcommand{\Re}{\operatorname{Re}}
\renewcommand{\Im}{\operatorname{Im}}

\newcommand{\<}{\langle}
\renewcommand{\>}{\rangle}

\newcommand{\from}{\colon}

\newcommand{\rfrac}[2]{\tfrac{#1}{\raisebox{0.1em}{\scriptsize$#2$}}}

\newcommand{\nr}[1]{\vspace{0.2ex}\noindent\hspace*{9mm}\llap{\textup{(#1)}}}

\renewcommand\le{\leqslant}
\renewcommand\ge{\geqslant}
\renewcommand{\|}{|\!|}
\renewcommand\c{\textnormal{c}}

\newcommand\ess{\textnormal{ess}}
\newcommand{\bigln}{\big|\mskip-4mu\mathopen\big|}
\newcommand{\bigrn}{\mathclose\big|\mskip-4mu\big|}

\newcommand{\set}[2]{\bigl\{#1{;}\breakOK\;#2\bigr\}}
\newcommand{\sset}[2]{\{#1{;}\breakOK\;#2\}}

\newcommand{\+}{\hspace{0.1em}}
\newcommand{\eps}{\varepsilon}

\newcommand\operator[1]{\expandafter\newcommand\csname#1\endcsname{\operatorname{#1}}}
\operator{vol}
\operator{spt}
\operator{dist}

\swapnumbers
\newtheorem{theorem}{Theorem}[section]
\newtheorem{corollary}[theorem]{Corollary}
\newtheorem{lemma}[theorem]{Lemma}
\newtheorem{proposition}[theorem]{Proposition}

\theoremstyle{definition}

\newtheorem{remark}[theorem]{Remark}
\newtheorem{remarks}[theorem]{Remarks}

\numberwithin{equation}{section}

\newcommand{\breakOK}{\penalty0}
\newcommand{\avoidbreak}{\postdisplaypenalty50}
\makeatletter
\renewcommand{\.}{.\nobreak\@ifnextchar1{\hskip0.18em plus0.05em minus
0.02em}{\hskip0.22em plus0.06em minus 0.04em}}
\renewcommand\section{\@startsection {section}{1}{\z@}%
                                     {-3.25ex \@plus -1ex \@minus -.2ex}%
                                     {1.5ex \@plus.2ex}%
                                     {\normalfont\large\bfseries}}
\def\env@cases{%
  \let\@ifnextchar\new@ifnextchar
  \left\lbrace
  \def\arraystretch{1.1}%
  \array{@{\,}l@{\quad}l@{}}%
}
\@ifundefined{MoveEqLeft}{%
\newcommand\MoveEqLeft[1][2]{%
  \global\@tempdima=#1em%
  \kern\@tempdima%
  &
  \kern-\@tempdima}
}{}

\newcommand*\if@single[3]{%
  \setbox0\hbox{${\mathaccent"0362{#1}}^H$}%
  \setbox2\hbox{${\mathaccent"0362{\kern0pt#1}}^H$}%
  \ifdim\ht0=\ht2 #3\else #2\fi
  }
\newcommand*\rel@kern[1]{\kern#1\dimexpr\macc@kerna}
\newcommand*\widebar[1]{\@ifnextchar^{{\wide@bar{#1}{0}}}{\wide@bar{#1}{1}}}
\newcommand*\wide@bar[2]{\if@single{#1}{\wide@bar@{#1}{#2}{1}}{\wide@bar@{#1}{#2}{2}}}
\newcommand*\wide@bar@[3]{%
  \begingroup
  \def\mathaccent##1##2{%
    \if#32 \let\macc@nucleus\first@char \fi
    \setbox\z@\hbox{$\macc@style{\macc@nucleus}_{}$}%
    \setbox\tw@\hbox{$\macc@style{\macc@nucleus}{}_{}$}%
    \dimen@\wd\tw@
    \advance\dimen@-\wd\z@
    \divide\dimen@ 3
    \@tempdima\wd\tw@
    \advance\@tempdima-\scriptspace
    \divide\@tempdima 10
    \advance\dimen@-\@tempdima
    \ifdim\dimen@>\z@ \dimen@0pt\fi
    \rel@kern{0.6}\kern-\dimen@
    \if#31
      \overline{\rel@kern{-0.6}\kern\dimen@\macc@nucleus
                \rel@kern{0.4}\kern\dimen@}%
      \advance\dimen@0.4\dimexpr\macc@kerna
      \let\final@kern#2%
      \ifdim\dimen@<\z@ \let\final@kern1\fi
      \if\final@kern1 \kern-\dimen@\fi
    \else
      \overline{\rel@kern{-0.6}\kern\dimen@#1}%
    \fi
  }%
  \macc@depth\@ne
  \let\math@bgroup\@empty \let\math@egroup\macc@set@skewchar
  \mathsurround\z@ \frozen@everymath{\mathgroup\macc@group\relax}%
  \macc@set@skewchar\relax
  \let\mathaccentV\macc@nested@a
  \if#31
    \macc@nested@a\relax111{#1}%
  \else
    \def\gobble@till@marker##1\endmarker{}%
    \futurelet\first@char\gobble@till@marker#1\endmarker
    \ifcat\noexpand\first@char A\else
      \def\first@char{}%
    \fi
    \macc@nested@a\relax111{\first@char}%
  \fi
  \endgroup
}
\makeatother

\begin{document}

\medmuskip=4mu plus 2mu minus 3mu
\thickmuskip=5mu plus 3mu minus 1mu
\abovedisplayshortskip=0pt plus 3pt minus 3pt
\belowdisplayshortskip=\belowdisplayskip
\newcommand\smallads{\vadjust{\vskip-5pt plus-1pt minus-1pt}}
\newcommand\smallbds{\vskip-1\lastskip\vskip5pt plus3pt minus2pt\noindent}

\title{\Large $L_1$-estimates for eigenfunctions
and heat kernel estimates \\for semigroups
dominated by the free heat semigroup}
\author{Hendrik Vogt%
\footnote{Fachbereich 3 -- Mathematik, Universit{\"a}t Bremen, 28359 Bremen, Germany,
+49\,421\,218-63702,
{\tt hendrik.vo\rlap{\textcolor{white}{hugo@egon}}gt@uni-\rlap{\textcolor{white}{hannover}}bremen.de}}}
\date{}

\maketitle

\begin{abstract}
We investigate selfadjoint positivity preserving $C_0$-semi\-groups that are
dominated by the free heat semigroup on $\R^d$.
Major examples are semigroups generated by Dirichlet Laplacians on open subsets
or by Schr\"odinger operators with absorption potentials.
We show explicit global Gaussian upper bounds for the kernel that correctly
reflect the exponential decay of the semigroup.
For eigenfunctions of the generator that correspond to eigenvalues below the
essential spectrum we prove estimates of their $L_1$-norm in terms of the
$L_2$-norm and the eigenvalue counting function.
This estimate is applied to a comparison of the heat content with the heat
trace of the semigroup.

\vspace{8pt}

\noindent
MSC 2010: 35P99, 35K08, 35J10, 47A10

\vspace{2pt}

\noindent
Keywords: Schr\"odinger operators, eigenfunctions, $L_1$-estimates, \\
\phantom{Keywords: }%
heat kernel estimates, heat content, heat trace
\end{abstract}

\section{Introduction and main results}\label{sec_intro}

In the recent paper \cite{bhv13}, the authors studied Dirichlet Laplacians on
open subsets $\Omega$ of $\R^d$. They proved an estimate for the $L_1$-norm of
eigenfunctions in terms of their $L_2$-norm and spectral data, and they used
this to estimate the heat content of $\Omega$ by its heat trace. The aim of the
present paper is to provide sharper estimates in the following more general
setting.

Let $\Omega\subseteq\R^d$ be measurable, where $d\in\N$, and let $T$ be a selfadjoint positivity
preserving $C_0$-semi\-group on $L_2(\Omega)$ that is dominated by the free heat
semigroup, i.e.,
\[
  0 \le T(t)f \le e^{t\Delta}f \qquad \bigl(t\ge0,\ 0 \le f \in L_2(\Omega)\bigr).
  \avoidbreak
\]
Let $-H$ denote the generator of $T$.

An important example for the operator $-H$ is the Dirichlet Laplacian with
a locally integrable absorption potential on an open set $\Omega\subseteq\R^d$.
For more general absorption potentials the space of strong continuity of
the semigroup will be $L_2(\Omega')$ for some measurable $\Omega'\subseteq\Omega$.

In our first main result we estimate the $L_1$-norm of eigenfunctions of $H$ in
terms of their $L_2$-norm and the eigenvalue counting function $N_t(H)$, which
for $t < \inf\sigma_\ess(H)$ denotes the number of eigenvalues of $H$ that are $\le
t$, counted with multiplicity.

\begin{theorem}\label{l1-est}
Let $\varphi$ be an eigenfunction of $H$ with eigenvalue $\lambda < \inf\sigma_\ess(H)$. Then
\[
  \|\varphi\|_1^2 \le c_d (t-\lambda)^{-d/2} \bigl(\ln\tfrac{3t}{t-\lambda}\bigr)^d N_t(H) \|\varphi\|_2^2
  \qquad \bigl(\lambda < t < \inf\sigma_\ess(H)\bigr),
\]
with $c_d = 35^{d+1}d^{d/2}$.
\end{theorem}

\begin{remarks}\label{l1-rem}
(a) We point out that as in \cite[Thm\.1.6]{bhv13} one has the lower bound
\[
  \|\varphi\|_1^2 \ge \bigl(\rfrac{2\pi d}{e}\bigr)^{d/2} \lambda^{-d/2} \|\varphi\|_2^2\,.
\]
Thus, the factor $d^{d/2}$ in the constant $c_d$ is of the correct order. The
factor $(t-\lambda)^{-d/2}$ matches the factor $\lambda^{-d/2}$; cf.\
Corollary~\ref{l1-cor} below. See \cite[Example~1.8(3)]{bhv13} for an
explanation why one should expect the factor $N_t(H)$ with some $t>\lambda$
in the estimate of Theorem~\ref{l1-est}.

(b) In \cite[Thm\.1.3]{bhv13}, in the framework of Dirichlet Laplacians on open
subsets of $\R^d$, the estimate
\[\renewcommand\-{\mkern-4.5mu}
  \|\varphi\|_1^2 \le C_d E_0^{-d/2} \left(
          \left( \frac{\lambda}{E_0} \+ \frac{\lambda}{t-\lambda} \right)^{\-d\!}
            \bigl(\ln N_t(H)\bigr)^d N_t(H)
          + \left( \frac{\lambda}{E_0} \right)^{\-4d-3\!}
            \left( \frac{\lambda}{t-\lambda} \right)^{\-4d\,}
          \right) \|\varphi\|_2^2
\]
was shown under the additional assumption $t\le3\lambda$, where $E_0 = \inf\sigma(H)$. Our
estimate $\|\varphi\|_1^2 \le c_d \lambda^{-d/2} \bigl(\frac{\lambda}{t-\lambda}\bigr)^{d/2} \bigl(\ln\tfrac{3t}{t-\lambda}\bigr)^d
N_t(H) \|\varphi\|_2^2$ improves on this in several regards; most notably, the factors
$\frac{\lambda}{E_0} $,\, $\bigl(\ln N_t(H)\bigr)^d$ and the second summand are removed
altogether.

(c) In \cite{bhv13}, a partition of $\R^d$ into cubes was used in the proof. We
will work with a ``continuous partition'' into balls instead; see the proof of
Lemma~\ref{E0-est}. Working with balls leads to a better constant $c_d$ in the
estimate.

(d) In the case $d=1$ and $H$ the Dirichlet Laplacian on an open subset of $\R$,
an improved estimate is given in \cite[Rem\.1.5]{bhv13}. For that estimate it
is crucial that $H$ is a direct sum of Dirichlet Laplacians on intervals. The
improvement is not possible for general $H$ in dimension $d=1$; this can be
seen similarly as in \cite[Example~1.8(3)]{bhv13}.
\end{remarks}

If $H$ has compact resolvent, then one can apply Theorem~\ref{l1-est} with
$t=(1+\eps)\lambda$ for any $\eps>0$ to obtain the following estimate. Note that it
contains the same factor $\lambda^{-d/2}$ as the lower bound of
Remark~\ref{l1-rem}(a).

\begin{corollary}\label{l1-cor}
Assume that $H$ has compact resolvent. Let $\varphi$ be an eigenfunction of $H$ with
eigenvalue $\lambda$. Then
\[
  \|\varphi\|_1^2 \le c_d C_\eps^d \lambda^{-d/2} N_{(1+\eps)\lambda}(H) \|\varphi\|_2^2 \qquad (\eps>0),
\]
with $c_d$ as in Theorem~\ref{l1-est} and $C_\eps = \eps^{-1/2}\ln(3+\rfrac3\eps)$.
\end{corollary}

\begin{remark}
The assumption that $H$ has compact resolvent is in particular satisfied if $\Omega$
has finite volume. Note that then the trivial estimate $\|\varphi\|_1^2 \le \vol(\Omega) \|\varphi\|_2^2$
holds. We point out that, up to a dimension dependent constant, the estimate of
Corollary~\ref{l1-cor} is never worse since one has the bound $N_t(H) \le
K_d\vol(\Omega)t^{d/2}$ for all $t>0$.
(To obtain this bound, apply \cite[Cor\.1]{liya83} to open sets $\tilde\Omega \supseteq \Omega$
and note that $e^{-tH} \le e^{-t\Delta_{\tilde\Omega}}$, where $\Delta_{\tilde\Omega}$ denotes the
Dirichlet Laplacian on $\tilde\Omega$.)
\end{remark}

Our second main result is the following heat kernel estimate for semigroups
dominated by the free heat semigroup. This estimate is obtained as a by-product
of the preparations for the proof of Theorem~\ref{l1-est}.

\begin{theorem}\label{hke}
For all $t>0$ the semigroup operator $e^{-tH}$ has an integral kernel~$p_t$.
If $E_0 := \inf\sigma(H) > 0$ then
\[
  0 \le p_t(x,y) \le \left( \frac{eE_0}{2\pi d} \right)^{\!d/2}
                 \exp\left( \!-E_0t-\frac{|x-y|^2}{4t} \right)
                 \qquad \bigl(t \ge \tfrac{d}{2E_0},\ x,y\in\R^d\bigr).
\]
\end{theorem}

\begin{remark}
(a) For $0 < t < \frac{d}{2E_0}$ one just has the estimate with respect to the
free heat kernel,
\[
  0 \le p_t(x,y) \le (4\pi t)^{-d/2} \smash[t]{\exp\left( \!-\frac{|x-y|^2}{4t} \right)}.
\]
In combination with Theorem~\ref{hke} this gives
\[
  0 \le p_t(x,y) \le (4\pi t)^{-d/2} \left( 1 + \frac{2e}{d} E_0t \right)^{\!d/2}
                 \exp\left( \!-E_0t-\frac{|x-y|^2}{4t} \right)
                 \qquad (t>0).
\]
(In the case $E_0=0$ this estimate is true but inconsequential.)

(b) In \cite[formula~(22)]{ouh06}, the following estimate was proved in the
framework of Dirichlet Laplacians with absorption potentials on open subsets of
$\R^d$:
\[
  p_t(x,y) \le c_\eps (4\pi t)^{-d/2}
             \left( 1 + \frac12 E_0t + \eps\frac{|x-y|^2}{8t} \right)^{\!d/2}
             \exp\left( \!-E_0t-\frac{|x-y|^2}{4t} \right),
\]
where $\eps>0$ and $c_\eps = e^2 \smash{(1+\rfrac1\eps)^{d/2}}$. Part~(a) shows that the
summand $\eps\smash{\frac{|x-y|^2}{8t}}$ is actually not needed, which may come as
a surprise.

(c) In the generality of our setting, the estimate provided in
Theorem~\ref{hke} is probably the best one can hope for. Suppose, for example,
that the semigroup~$T$ is irreducible and that $E_0$ is an isolated eigenvalue
of $H$. Then the large time behaviour of $p_t$ is known:
\[
  e^{E_0t} p_t(x,y) \to \varphi(x)\varphi(y) \qquad (t\to\infty),
\]
where $\varphi$ is the non-negative normalized ground state of $H$; see, e.g.,
\cite[Thm\.3.1]{klvw13}. Moreover, if $\inf\sigma(H)=1$ then $E^{d/4} \varphi(E^{1/2}\cdot)$ is the
ground state of an appropriately scaled operator $H_E$ with $\inf\sigma(H_E)=E$.
This explains the factor $E_0^{d/2}$ in our estimate.

Note, however, that better estimates are known for Dirichlet Laplacians under
suitable geometric assumptions on the domain $\Omega$. Then a boundary term like
$\varphi(x)\varphi(y)$ can be included in the estimate. This can be shown via intrinsic
ultracontractivity as in \cite{ouwa07}.
\end{remark}

An important application of Corollary~\ref{l1-cor} is that it allows us to
compare the ``heat content'' of $H$ with its ``heat trace''. We assume that $H$
has compact resolvent, with $(\lambda_k)$ the increasing sequence of all the
eigenvalues of $H$, repeated according to their multiplicity. For $t>0$ we
denote by $Q_H(t) := \|e^{-tH}\ind_\Omega\|_1$ the \emph{heat content}, by $Z_H(t) :=
\sum_{k=1}^\infty e^{-t\lambda_k}$ the \emph{heat trace} of~$H$.

Note that $Q_H$, $Z_H$ are decreasing functions.
It may well occur that $Q_H(t) = \infty$ and/or $Z_H(t) = \infty$
for some but not all $t>0$ if $\Omega$ has infinite Lebesgue measure,
see \cite[Thm\.5.5]{beda89}.

\begin{theorem}\label{heat-trace}
Assume that $H$ has compact resolvent and that $Z_H(t_0) < \infty$ for some $t_0>0$.
Then $Q_H(t) < \infty$ for all $t>2t_0$,
\[
  Q_H(t) \le c_{\eps,d} \lambda_1^{-d/2} Z_H\bigl(\tfrac{t}{2+\eps}\bigr)^2 \qquad \bigl(0<\eps<\tfrac{t}{t_0}-2\bigr),
\]
with $c_{\eps,d} = c_d C_\eps^d$ as in Corollary~\ref{l1-cor}.
\end{theorem}

The proof is rather short, so we give it right here. We will use the following
simple estimate.

\begin{lemma}\label{Nlambda}
(cf.\ \cite[Lemma~5.2]{bhv13})
For $T,\lambda>0$ one has $N_\lambda(H) \le Z_H(T)e^{T\lambda}$.
\end{lemma}
\begin{proof}
If $k\in\N$ is such that $\lambda_k\le\lambda$, then $k \le e^{T\lambda_k} \sum_{j=1}^k e^{-T\lambda_j} \le e^{T\lambda}
Z_H(T)$. Thus, $N_\lambda(H) = \#\sset{k}{\lambda_k\le\lambda} \le e^{T\lambda} Z_H(T)$.
\end{proof}

\begin{proof}[Proof of Theorem~\ref{heat-trace}]
Let $T := \frac{t}{2+\eps}$. Let $(\varphi_k)$ be an orthonormal basis of $L_2(\Omega)$ such
that $H\varphi_k = \lambda_k\varphi_k$ for all $k\in\N$. By Corollary~\ref{l1-cor} and
Lemma~\ref{Nlambda} we obtain
\[
  \|\varphi_k\|_1^2 \le c_{\eps,d} \lambda_k^{-d/2} N_{(1+\eps)\lambda_k}(H) \|\varphi_k\|_2^2
         \le c_{\eps,d} \lambda_1^{-d/2} Z_H(T) e^{T(1+\eps)\lambda_k}
\]
for all $k\in\N$. For $f\in L_2(\Omega)\cap L_\infty(\Omega)$ one has $e^{-tH}f = \sum_{k=1}^\infty \< f,\varphi_k\>
e^{-t\lambda_k}\varphi_k$ and hence
\[
  \|e^{-tH}f\|_1 \le \sum_{k=1}^\infty \|f\|_\infty e^{-t\lambda_k} \|\varphi_k\|_1^2\,.
\]
Using a sequence $(f_k)$ in $L_2(\Omega)$ with $0 \le f_k \uparrow \ind_\Omega$ and recalling $T(1+\eps) =
t-T$, we conclude that
\[
  \|e^{-tH}\ind_\Omega\|_1 \le \sum_{k=1}^\infty e^{-t\lambda_k} \|\varphi_k\|_1^2
                \le c_{\eps,d} \lambda_1^{-d/2} Z_H(T) \sum_{k=1}^\infty e^{-T\lambda_k}
                = c_{\eps,d} \lambda_1^{-d/2} Z_H(T)^2. \qedhere
\]
\end{proof}

The paper is organized as follows. In Section~\ref{sec_sgs} we investigate
properties of selfadjoint positivity preserving semigroups dominated by the
free heat semigroup. In Section~\ref{sec_hke} we prove Theorem~\ref{hke}, and
we show off-diagonal resolvent estimates needed in the proof of
Theorem~\ref{l1-est}, which in turn is given in Section~\ref{sec_l1-est}.

\section{Semigroups dominated by the free heat semigroup}\label{sec_sgs}

Throughout this section let $\Omega\subseteq\R^d$ be measurable, and let $T$ be a selfadjoint
positivity preserving $C_0$-semi\-group on $L_2(\Omega)$ that is dominated by the free
heat semigroup, with generator~$-H$. Let $\tau$ be the closed symmetric form
associated with $H$. The purpose of this section is to collect some basic
properties of $\tau$ and~$H$.

It is crucial that $D(\tau)$ is a subset of $H^1(\R^d)$ (in fact an ideal; see, e.g.,
\cite[Cor\.4.3]{mvv05}). Thus we can define a symmetric form $\sigma$ by
\begin{equation}\label{sigma-def}
  \sigma(u,v) := \tau(u,v) - \<\nabla u,\nabla v\> \qquad \bigl(u,v \in D(\sigma) := D(\tau)\bigr).
\end{equation}
This gives a decomposition of the form $\tau$ as the standard Dirichlet form plus
a form $\sigma$ that is positive and local in the sense of the following lemma. If
$-H$ is the Dirichlet Laplacian with an absorption potential $V\ge0$ on an open
set $\Omega\subseteq\R^d$, then $\sigma(u,v) = \int Vu\widebar{v}$. In this case the next three
results are trivial.

\begin{lemma}\label{sigma-pos}
Let $0 \le u,v \in D(\tau)$. Then $\sigma(u,v) \ge 0$, and $\sigma(u,v)=0$ if $u\wedge v=0$.
\end{lemma}
\begin{proof}
By \cite[Cor\.4.3]{mvv05}, the first assertion follows from the assumption that
$T$ is a positive semigroup dominated by the free heat semigroup.  For the
second assertion let $w := u-v$. Then $\tau(u,v) = \tau(w^+,w^-) \le 0$ since $T$ is a
positive semigroup (see, e.g., \cite[Cor\.2.6]{mvv05}). Since $\<\nabla u,\nabla v\> = 0$,
this implies $\sigma(u,v) \le 0$ and hence $\sigma(u,v)=0$.
\end{proof}

\begin{lemma}\label{multiplicator}
If $\xi \in W_{\!\infty}^1(\R^d)$ and $u \in D(\tau)$, then $\xi u \in D(\tau)$. Moreover, $f \from \R^d
\to D(\tau)$,\, $f(x) := \xi(\cdot-x)u$ is continuous.
\end{lemma}
\begin{proof}
By \cite[Cor\.4.3]{mvv05}, $D(\tau)$ is an ideal of $H^1(\R^d)$. This implies the
first assertion $\xi u \in D(\tau)$ since $\xi u \in H^1(\R^d)$ and $|\xi u| \le \|\xi\|_\infty |u| \in
D(\tau)$.

For the second assertion it suffices to show continuity at $0$, and we can
assume without loss of generality that $\xi$, $u$ are real-valued. From the
identity
\[
  f(x) - f(0) = \xi(\cdot-x)\bigl(u-u(\cdot-x)\bigr) + (\xi u)(\cdot-x) - \xi u
\]
one deduces that $f \from \R^d \to H^1(\R^d)$ is continuous at $0$. By
Lemma~\ref{sigma-pos} we obtain
\[
  \sigma\bigl(f(x)-f(0)\bigr) = \sigma\bigl(|f(x)-f(0)|\bigr) \le \sigma\bigl(\|\xi(\cdot-x)-\xi\|_\infty|u|\bigr) \le \|\nabla\xi\|_\infty^2 |x|^2 \sigma(|u|).
  \avoidbreak
\]
Due to the decomposition~\eqref{sigma-def} this yields continuity of $f \from
\R^d \to D(\tau)$ at $0$.
\end{proof}

\begin{lemma}\label{sigma}
Let $u,v\in D(\tau)$. Then $\sigma(\xi u,v) = \sigma(u,\xi v)$ for all $\xi \in W_{\!\infty}^1(\R^d)$.
\end{lemma}
\begin{proof}
Since $D(\tau)$ is a lattice, it suffices to show the assertion for $u,v\ge0$ and
real-valued $\xi$. Throughout the proof we consider only real-valued function
spaces. We define a bilinear form $b$ by
\[
  b(\varphi,\psi) := \sigma(\varphi u,\psi v) \qquad \bigl(\varphi,\psi \in D(b) := W_{\!\infty,0}^1(\R^d)\bigr).
\]
Then $b(\varphi,\psi)\ge0$ for $\varphi,\psi\ge0$ by Lemma~\ref{sigma-pos}. Now one can proceed
similarly as in \cite[proof of Thm\.4.1]{arwa03} to show that
\begin{equation}\label{b-repr}
  \sigma(\varphi u,\psi v) = \int \varphi \psi\,d\mu \qquad \bigl(\varphi,\psi \in W_{\!\infty,0}^1(\R^d)\bigr)
\end{equation}
for some finite positive Borel measure $\mu$ on $\R^d$ (depending of course on
$u,v$). We only sketch the argument: first one can extend $b$ to a continuous
bilinear form on $C_0(\R^d)$, by positivity. Then one uses the linearisation of
$b$ in $C_0(\R^d\times \R^d)'$ to obtain a finite Borel measure $\nu$ on $\R^d\times \R^d$ such that
$b(\varphi,\psi) = \int \varphi(x) \psi(y)\,d\nu(x,y)$ for all $\varphi,\psi \in W_{\!\infty,0}^1(\R^d)$. Finally, $\spt\nu
\subseteq \set{(x,x)}{x\in\R^d}$ since $b(\varphi,\psi)=0$ in the case $\spt\varphi \cap \spt\psi = \varnothing$, by
Lemma~\ref{sigma-pos}, and this leads to the asserted measure $\mu$.

To complete the proof, we show that the representation~\eqref{b-repr} is valid
for all $\varphi,\psi \in W_{\!\infty}^1(\R^d)$. Let $\chi\in C_\c^1(\R^d)$ such that $0\le\chi\le1$ and
$\chi|_{B(0,1)} = 1$. Then $u_n := \chi(\rfrac\cdot n)u \to u$ in $H^1(\R^d)$ as $n\to\infty$, and
$\sigma(u_n) \le \sigma(u)$ for all $n\in\N$ by Lemma~\ref{sigma-pos}. Therefore, $\limsup
\tau(u_n) \le \tau(u)$, and this implies $u_n\to u$ in $D(\tau)$. Applying~\eqref{b-repr} to
$\sigma\bigl(\chi(\rfrac\cdot n)\varphi u,\chi(\rfrac\cdot n)\psi v\bigr)$ and letting $n\to\infty$ we derive~\eqref{b-repr}
for any $\varphi,\psi \in W_{\!\infty}^1(\R^d)$. For real-valued $\xi \in W_{\!\infty}^1(\R^d)$ we now obtain
\[
  \sigma(\xi u,v) = \int \xi\,d\mu= \sigma(u,\xi v). \qedhere
\]
\end{proof}

In the proof of Theorem~\ref{l1-est} we will work with operators that are
subordinated to $H$ as follows. For an open set $U\subseteq\R^d$ let $H_U$ denote the
selfadjoint operator in $L_2(\Omega\cap U)$ associated with the form $\tau$ restricted to
$D(\tau)\cap H_0^1(U)$. (Observe that this form domain is dense in $L_2(\Omega\cap U)$.)

\begin{lemma}\label{HG-action}
Let $\varphi$ be an eigenfunction of $H$ with eigenvalue $\lambda$. Let $U$ be an open
subset of $\R^d$, and let $\xi \in W_{\!\infty}^2(\R^d)$,\, $\xi=0$ on $\R^d\setminus U$. Then $\xi\varphi \in
D(H_U)$ and
\[
  (H_U-\lambda)(\xi\varphi) = -2\nabla\xi\cdot\nabla\varphi - (\Delta\xi)\varphi.
\]
\end{lemma}
\begin{proof}
By Lemma~\ref{multiplicator} we have $\xi\varphi \in D(\tau)$.
Moreover, $\xi\varphi \in H_0^1(U)$ due to the assumption $\xi=0$ on $\R^d\setminus U$.
For $v \in D(\tau)\cap H_0^1(U)$ we have $\xi v \in D(\tau)\cap H_0^1(U)$ and
\[
  (\tau-\lambda)(\varphi,\xi v) = \< (H-\lambda)\varphi, \xi v \> = 0.
\]
Since $\sigma(\xi\varphi,v) = \sigma(\varphi,\xi v)$ by Lemma~\ref{sigma}, the
decomposition~\eqref{sigma-def} yields
\begin{align*}
(\tau-\lambda)(\xi\varphi,v) &= (\tau-\lambda)(\xi\varphi,v) - (\tau-\lambda)(\varphi,\xi v) \\
            &= \<\nabla(\xi\varphi),\nabla v\> - \<\nabla\varphi,\nabla(\xi v)\> = \<\varphi\nabla\xi,\nabla v\> - \<\nabla\varphi,v\nabla\xi\>.
\end{align*}
Now $\varphi\nabla\xi$ is in $H^1(\R^d)^d$ and $\nabla\cdot(\varphi\nabla\xi) = \nabla\varphi\cdot\nabla\xi + \varphi\Delta\xi$, so we conclude that
\[
  (\tau-\lambda)(\xi\varphi,v) = -\<2\nabla\varphi\cdot\nabla\xi+\varphi\Delta\xi,v\>
\]
for all $v \in D(\tau)\cap H_0^1(U)$, which proves the assertion.
\end{proof}

\section{Heat kernel estimates}\label{sec_hke}

In this section we prove Theorem~\ref{hke}, and we provide resolvent estimates
needed in the proof of Theorem~\ref{l1-est}. Throughout we denote
\[
  \C_+ := \sset{z\in\C}{\Re z > 0}.
\] We point out that in the following result $T$ is not required to be a
semigroup.

\begin{proposition}\label{weighted-est}
Let $(\Omega,\mu)$ be a measure space, and let $\rho \from \Omega\to\R$ be measurable. Let $\lambda\in\R$,
and let $T \from \C_+\to\Ls(L_2(\mu))$ be analytic, $\|T(z)\| \le e^{-\lambda\Re z}$ for all
$z\in\C_+$. Assume that there exists $C\ge0$ such that
\[
  \|e^{\alpha\rho}T(t)e^{-\alpha\rho}\| \le Ce^{\alpha^2t} \qquad (\alpha,t>0).
\]\smallbds
Then
\[
  \|e^{\alpha\rho}T(z)e^{-\alpha\rho}\| \le \exp\bigl(\alpha^2/\Re\rfrac1z-\lambda\Re z\bigr) \qquad (\alpha>0,\ z\in\C_+),
\]
in particular, $\|e^{\alpha\rho}T(t)e^{-\alpha\rho}\| \le e^{\alpha^2t-\lambda t}$ for all $\alpha,t>0$.
\end{proposition}

Here and in the following we denote
\[
  \|wBw^{-1}\| := \sup\set{\|wBw^{-1}f\|_2}{f\in L_2(\mu),\ \|f\|_2\le1,\ w^{-1}f\in L_2(\mu)}
\]
for an operator $B \in \Ls(L_2(\mu))$ and a measurable function $w \from \Omega\to(0,\infty)$.

\begin{proof}[Proof of Proposition~\ref{weighted-est}]
Observe that
\[
  M := \set{f\in L_2(\mu)}{\rho\ \text{bounded on}\ [f\ne0]}
\]
is dense in $L_2(\mu)$. Let $\alpha>0$, and let $f,g\in M$ with $\|f\|_2 = \|g\|_2 = 1$. Define
the analytic function $F \from \C_+\to\C$ by
\[
  F(z) := e^{\lambda z-\alpha^2/z} \< e^{\alpha\rho/z} T(z) e^{-\alpha\rho/z} f, g \> .
\]
Let $c>0$ such that $|\rho|\le c$ on $[f\ne0]\cup[g\ne0]$. Then
\[
  |F(z)| \le \|e^{\lambda z}T(z)\| \|e^{-\alpha\rho/z}f\|_2 \|e^{\alpha\rho/\bar z}g\|_2
         \le \exp(2\alpha c\Re\rfrac1z) \qquad (z\in\C_+),
\]
in particular $|F(t)| \le e^{2\alpha c}$ for all $t\ge1$. Moreover,
\[
  |F(t)| \le e^{\lambda t-\alpha^2/t} \|e^{\alpha\rho/t} T(t) e^{-\alpha\rho/t}\| \le e^{\lambda t-\alpha^2/t}\cdot Ce^{(\alpha/t)^2t} \le C e^{|\lambda|}
\]
for all $0<t<1$. Thus, $|F(z)| \le 1$ for all $z\in\C_+$ by the next
lemma, and this yields
\[
  \|e^{\alpha\rho/z}T(z)e^{-\alpha\rho/z}\| \le \exp\bigl(\alpha^2\Re\rfrac1z-\lambda\Re z\bigr) \qquad (\alpha>0,\ z\in\C_+).
\]
The assertion follows by replacing $\alpha$ with $\alpha/\Re\rfrac1z$.
\end{proof}

The following Phragm\'en-Lindel\"of type result is similar to
\cite[Prop\.2.2]{cosi08}.

\begin{lemma}\label{phl}
Let $F \from \C_+\to\C$ be analytic. Assume that there exist $c_1,c_2>0$ such that
\[
  |F(z)| \le \exp(c_1\Re\rfrac1z) \quad (z\in\C_+), \quad\qquad
  |F(t)| \le c_2                  \quad (t>0).
\]
Then $|F(z)| \le 1$ for all $z\in\C_+$.
\end{lemma}
\begin{proof}
Note that $\limsup_{z\to iy} |F(z)| \le 1$ for all $y\in\R\setminus\{0\}$.
Thus, $|F(z)| \le c_2\vee1$ for all $z\in\C_+$ by the
Phragm\'en-Lindel\"of principle applied to the sectors
$\set{z\in\C}{\Re z>0,\ \breakOK \Im z>0}$ and
$\set{z\in\C}{\Re z>0,\ \breakOK \Im z<0}$. Then an application of
the Phragm\'en-Lindel\"of principle to the sector $\C_+$ implies $|F| \le 1$ on
$\C_+$.
\end{proof}

In the next lemma we state a version of the well-known Davies' trick; cf.\
\cite[proof of Lemma~19]{dav95}. For the proof note that $\inf_{\xi\in\R^d}
\exp\bigl(|\xi|^2t-\xi\cdot x\bigr) = \exp\bigl(-\frac{|x|^2}{4t}\bigr)$ for all $t>0$,\, $x\in\R^d$.

\begin{lemma}\label{davies}
Let $\Omega\subseteq\R^d$ be measurable, and let $B$ be a positive operator on $L_2(\Omega)$. For
$\xi,x\in\R^d$ let $\rho_\xi(x) := e^{\xi x}$. Then for $t>0$ the following are equivalent:

\nr{i} $B \le e^{t\Delta}$,

\nr{ii} $\|\rho_\xi B\rho_\xi^{-1}\|_{1\to\infty} \le (4\pi t)^{-d/2} e^{|\xi|^2t}$ for all $\xi\in\R^d$.
\end{lemma}

In (i), the inequality $B \le e^{t\Delta}$ is meant in the sense of positivity
preserving operators, i.e.,
\[
  Bf \le e^{t\Delta}f \qquad \bigl(0 \le f \in L_2(\Omega)\bigr).
\]

The following result provides an estimate of the resolvent of $H$ by the free
resolvent. Together with Proposition~\ref{free-res} below this will be an
important stepping stone in the proof of Theorem~\ref{l1-est}.

\begin{theorem}\label{res-est}
Let $\Omega\subseteq\R^d$ be measurable, and let $T$ be a selfadjoint positive
$C_0$-semi\-group on $L_2(\Omega)$ that is dominated by the free heat semigroup. Let
$-H$ be the generator of\/ $T$, and let $E_0 := \inf \sigma(H)$. Then for all
$\eps\in(0,1]$ one has
\begin{align*}
  T(t) &\le \eps^{-d/2} e^{-(1-\eps)E_0t} e^{t\Delta} \qquad & (t\ge0), \\
(H-\lambda)^{-1} &\le \eps^{-d/2} \bigl((1-\eps)E_0-\lambda-\Delta\bigr)^{-1} & (\lambda < (1-\eps)E_0).
\end{align*}
\end{theorem}
\begin{proof}
As above let $\rho_\xi(x) := e^{\xi x}$. The assumptions imply $\|T(z)\|_{2\to2} \le e^{-E_0\Re z}$
for all $z\in\C_+$ and
\[
  \|\rho_\xi T(t) \rho_\xi^{-1}\|_{2\to2} \le \|\rho_\xi e^{t\Delta} \rho_\xi^{-1}\|_{2\to2} = e^{t|\xi|^2}
  \qquad (\xi\in\R^d,\ t\ge0).
\]
By Proposition~\ref{weighted-est} it follows that
\begin{equation}\label{2-2-est}
  \|\rho_\xi T(t) \rho_\xi^{-1}\|_{2\to2} \le e^{t|\xi|^2-E_0t} \qquad (\xi\in\R^d,\ t\ge0).
\end{equation}

Let $t>0$, and let $k_t$ be the convolution kernel of $e^{t\Delta}$. Then for $\xi\in\R^d$
the kernel of $e^{-t|\xi|^2} \rho_\xi e^{t\Delta} \rho_\xi^{-1}$ is given by
\[
  e^{-t|\xi|^2+\xi\cdot(x-y)} k_t(x-y) = k_t(x-2t\xi-y) \qquad (x,y\in\R^d)
\]
since $-t|\xi|^2+\xi\cdot(x-y)-\frac{|x-y|^2}{4t} = -\frac{|x-y-2t\xi|^2}{4t}$. (The above
identity is the key point in the proof; this is why we need unbounded weights
in Proposition~\ref{weighted-est}.) Therefore,
\[
  e^{-t|\xi|^2} \|\rho_\xi T(t) \rho_\xi^{-1}\|_{2\to\infty}
  \le \|e^{t\Delta}\|_{2\to\infty} = \|e^{2t\Delta}\|_{1\to\infty}^{1/2} = (8\pi t)^{-d/4}.
\]
By duality we also have $e^{-t|\xi|^2} \|\rho_\xi T(t) \rho_\xi^{-1}\|_{1\to2} \le (8\pi t)^{-d/4}$.
Using the semigroup property and~\eqref{2-2-est}, we conclude for $\eps\in(0,1]$
that
\begin{align*}
\|\rho_\xi T(t) \rho_\xi^{-1}\|_{1\to\infty}
 &\le \|\rho_\xi T(\tfrac\eps2 t) \rho_\xi^{-1}\|_{2\to\infty} \|\rho_\xi T((1-\eps)t) \rho_\xi^{-1}\|_{2\to2}
    \|\rho_\xi T(\tfrac\eps2 t) \rho_\xi^{-1}\|_{1\to2} \\
 &\le e^{t|\xi|^2} (8\pi\tfrac\eps2 t)^{-d/4} e^{-E_0(1-\eps)t} (8\pi\tfrac\eps2 t)^{-d/4}
  = (4\pi\eps t)^{-d/2} e^{t|\xi|^2-E_0(1-\eps)t}.
\end{align*}
Now the first assertion follows from Lemma~\ref{davies}, and this gives the
second assertion by the resolvent formula.
\end{proof}

\begin{proof}[Proof of Theorem~\ref{hke}]
The existence of the kernel $p_t$ follows from the Dunford-Pettis theorem, and
Theorem~\ref{res-est} implies
\[
  p_t(x,y)
  \le (4\pi\eps t)^{-d/2} e^{\eps E_0t} \exp\left( \!-E_0t-\frac{|x-y|^2}{4t} \right)
\]
for all $t>0$. Then for $t \ge \frac{d}{2E_0}$ the assertion follows by setting $\eps
:= \frac{d}{2E_0t}$.
\end{proof}

We conclude this section with an off-diagonal $L_1$-estimate for the free
resolvent.

\begin{proposition}\label{free-res}
Let $A,B\subseteq\R^d$ be measurable, and let $d(A,B)$ denote the distance between $A$
and $B$. Then
\[
  \|\ind_A(\mu-\Delta)^{-1}\ind_B\|_{1\to1} \le (1-\theta^2)^{-d/2} \rfrac1\mu \exp\bigl(-\theta\sqrt\mu\,d(A,B)\bigr)
\]
for all $\mu>0$,\, $0<\theta<1$.
\end{proposition}
\begin{proof}
Let $r:=d(A,B)$. By duality we have to show
\[
  \|\ind_B(\mu-\Delta)^{-1}\ind_A\|_{\infty\to\infty} \le (1-\theta^2)^{-d/2} \rfrac1\mu \exp\bigl(-\theta r\sqrt\mu\bigr) =: C,
\]
or equivalently, $(\mu-\Delta)^{-1}\ind_A \le C$ on $B$. Let $x\in B$. By the resolvent formula
we obtain
\[
  \renewcommand{\-}{\hspace{-0.1em}}
  (\mu-\Delta)^{-1}\ind_A(x) = \int_0^\infty \- e^{-\mu t} \- \int_{\R^d} k_t(y) \ind_A(x-y)\,dy\,dt
               \le \int_0^\infty \- e^{-\mu t} \- \int_{|y|\ge r} k_t(y)\,dy\,dt,
\]
where $k_t(y) = (4\pi t)^{-d/2} \exp\bigl(-\frac{|y|^2}{4t}\bigr)$. We substitute $y = (4t)^{1/2}
z$ and note that $|y|\ge r$ if and only if $t \ge \bigl(\frac{r}{2|z|}\bigr)^2$; then by
Fubini's theorem we infer that
\begin{align*}
(\mu-\Delta)^{-1}\ind_A(x)
 &\le \pi^{-d/2} \int_{\R^d}
    \int_{\left(\frac{r}{2|z|}\right)^2}^\infty e^{-\mu t}\,dt\, e^{-|z|^2}\,dz \\
 &= \pi^{-d/2} \frac1\mu \int_{\R^d} \exp\!\left( -\frac{\mu r^2}{4|z|^2} - |z|^2 \right)dz.
\end{align*}
Note that $\theta r\sqrt\mu \le \frac{\mu r^2}{4|z|^2} + \theta^2|z|^2$ and hence $\exp\bigl( -\frac{\mu
r^2}{4|z|^2} - |z|^2 \bigr) \le e^{-\theta r\sqrt\mu} e^{-(1-\theta^2)|z|^2}$ for all $z\in\R^d$. We conclude
that
\[
  (\mu-\Delta)^{-1}\ind_A(x) \le \rfrac1\mu e^{-\theta r\sqrt\mu} \,\pi^{-d/2}\! \int_{\R^d} e^{-(1-\theta^2)|z|^2} dz
               = \rfrac1\mu e^{-\theta r\sqrt\mu} (1-\theta^2)^{-d/2},
\]
which proves the assertion.
\end{proof}

\begin{remark}
For $\mu > \bigl(\rfrac dr\bigr)^2$ (where $r=d(A,B)$), optimizing the estimate of
Proposition~\ref{free-res} with respect to $\theta$ leads to the choice $\theta =
\bigl(1-\frac{d}{r\sqrt\mu}\bigr)^{1/2}$. For $\mu > \bigl(\frac{d}{2r}\bigr)^2$, the choice $\theta =
1-\smash{\frac{d}{2r\sqrt\mu}}$ yields
\[
  \|\ind_A(\mu-\Delta)^{-1}\ind_B\|_{1\to1} \le \bigl(\tfrac{2e}{d}+r\sqrt\mu\bigr)^{d/2} \rfrac1\mu e^{-r\sqrt\mu}.
\]
\end{remark}

\section{Proof of Theorem~\ref{l1-est}}\label{sec_l1-est}

Throughout this section we assume the setting of Section~\ref{sec_sgs}, i.e.,
$\Omega\subseteq\R^d$ is measurable, $T$ a selfadjoint positivity preserving $C_0$-semi\-group
on $L_2(\Omega)$ dominated by the free heat semigroup, with generator $-H$, and $\tau$
the closed symmetric form associated with $H$. We denote
\[
  E_0(H) := \inf\sigma(H).
\]
Recall that, for an open set $U\subseteq\R^d$,\, $H_U$ is the selfadjoint operator in
$L_2(\Omega\cap U)$ associated with the form $\tau$ restricted to $D(\tau)\cap H_0^1(U)$.

For $A\subseteq\R^d$ we denote by $U_\eps(A) = \bigcup_{x\in A} B(x,\eps)$ the $\eps$-neighborhood of
$A$. If $A$ is measurable, then we write $|A|$ for the Lebesgue measure of $A$.
For $r>0$ and $E_0(H) < t < \inf\sigma_\ess(H)$ we define the sets
\begin{equation}\label{FrGr}
\begin{split}
  F_r(t) &:= \set{x\in\R^d}{E_0(H_{B(x,r)}) < t},
  \\[0.5ex plus 0.2ex minus 0.2ex]
  G_r(t) &:= \R^d \setminus \widebar{U_r(F_r(t))}.
\end{split}
\end{equation}
For the proof of Theorem~\ref{l1-est} the following two facts will be crucial.
On the one hand, the set $F_r(t)$ is ``small'' in the sense that the Lebesgue
measure of $U_{3r}(F_r(t))$ is not too large, as is expressed in the next lemma.
On the other hand, the set $G_r(t)$ is ``spectrally small'' in the sense that
the ground state energy of $H_{G_r(t)}$ is not much smaller than $t$; see
Lemma~\ref{E0-est} below.

\begin{lemma}\label{UsFr}
Let $r>0$ and $E_0(H) < t < \inf\sigma_\ess(H)$. Then
\[
  |U_s(F_r(t))| \le \omega_d (2r+s)^d N_t(H) \qquad (s>0),
\]
where $\omega_d := |B(0,1)|$.
\end{lemma}
\begin{proof}
Let $M \subseteq F_r(t)$ be a maximal subset with the property that the balls
$B(x,r)$,\, $x \in M$ are pairwise disjoint. Then by the min-max principle and
the definition of $F_r(t)$ one sees that $M$ has at most $N_t(H)$ elements.
Moreover, $F_r(t) \subseteq \bigcup_{x\in M} B(x,2r)$ by the maximality of $M$. Therefore,
\[
  |U_s(F_r(t))| \le \sum_{x\in M} |B(x,2r+s)| \le N_t(H) \cdot \omega_d(2r+s)^d. \qedhere
\]
\end{proof}

\begin{lemma}\label{E0-est}
Let $E_{0,d}$ denote the ground state energy of the Dirichlet Laplacian on
$B(0,1)$. Then $E_{0,d} \le \frac12(d+1)(d+2) \le \frac34(d+1)^2$, and
\[
  E_0(H_{G_r(t)}) \ge t - E_{0,d}/r^2 \qquad \bigl(r>0,\ E_0(H) < t < \inf\sigma_\ess(H)\bigr).
\]
\end{lemma}
\begin{proof}
For $\psi \in W_{2,0}^1\bigl(B(0,1)\bigr)$ defined by $\psi(x) = 1-|x|$ one easily computes
$\|\nabla\psi\|_2^2/\|\psi\|_2^2 = \frac12(d+1)(d+2)$, thus proving the first assertion. Let now $\psi$
denote the normalized ground state of the Dirichlet Laplacian on $B(0,1)$. For
$r>0$ let $\psi_r := r^{-d/2} \psi(\rfrac\cdot r)$; note that $\|\psi_r\|_2 = 1$ and $\psi_r \in
W_{\!\infty}^1(\R^d)$.

To prove the second assertion, we need to show that
\begin{equation}\label{tau-u}
  \tau(u) \ge \bigl(t - E_{0,d}/r^2\bigr) \|u\|_2^2
\end{equation}
for all $u \in D(\tau)\cap H_0^1(G_r(t))$, without loss of generality $u$ real-valued. We
will use $\bigl(\psi_r(\cdot-x)^2\bigr)\rule[-0.6ex]{0pt}{0pt}_{x\in\R^d}$ as a continuous partition
of the identity. By Lemma~\ref{multiplicator} we have $\psi_r(\cdot-x)u \in D(\tau)$ for all
$x\in\R^d$. Using~\eqref{sigma-def} and Lemma~\ref{sigma} we obtain
\begin{align*}
\tau\bigl(\psi_r(\cdot-x)u\bigr) &= \bigln \psi_r(\cdot-x)\nabla u + u\nabla\psi_r(\cdot-x) \bigrn_2^2 + \sigma\bigl(\psi_r(\cdot-x)u\bigr) \\
            &= \int \Bigl( \nabla\bigl(\psi_r(\cdot-x)^2u\bigr) \cdot \nabla u  + u^2|\nabla\psi_r(\cdot-x)|^2 \Bigr)
               + \sigma\bigl(\psi_r(\cdot-x)^2u,u\bigr) \\
            &= \tau\bigl(\psi_r(\cdot-x)^2u,u\bigr) + \int u^2|\nabla\psi_r(\cdot-x)|^2.
\end{align*}

Note that $\int \psi_r(y-x)^2\,dx = \|\psi_r\|_2^2 = 1$ and
\[
  \int |\nabla\psi_r(y-x)|^2\,dx = \|\nabla\psi_r\|_2^2 = \|\nabla\psi\|_2^2/r^2 = E_{0,d}/r^2
\]
for all $y\in\R^d$. Taking into account Lemma~\ref{multiplicator} (with $\xi = \psi_r^2$)
we thus obtain $\int \tau\bigl(\psi_r(\cdot-x)^2u,u\bigr)\,dx = \tau(u,u)$ and hence
\[
  \int \tau\bigl(\psi_r(\cdot-x)u\bigr)\,dx = \tau(u) + \|u\|_2^2\cdot E_{0,d}/r^2.
\]
To conclude the proof of~\eqref{tau-u}, we show that the left hand side of this
identity is greater or equal $t\|u\|_2^2$: note that $\psi_r(\cdot-x)u \in H_0^1(B(x,r))$. For
$x \in \R^d\setminus F_r(t)$ we have $\tau\bigl(\psi_r(\cdot-x)u\bigr) \ge t\|\psi_r(\cdot-x)u\|_2^2$ by the definition of
$F_r(t)$; for $x \in F_r(t)$ we have $\psi_r(\cdot-x)u = 0$ since $u \in H_0^1(G_r(t))$.
Therefore,
\[
  \int \tau\bigl(\psi_r(\cdot-x)u\bigr)\,dx \ge t \int \|\psi_r(\cdot-x)u\|_2^2\,dx = t\|u\|_2^2\,. \qedhere
\]
\end{proof}

\begin{remark}
It is known that $E_{0,d}$ behaves like $\frac14 d^2$ for large $d$. For $d=3$,
however, the estimate $E_{0,d} \le \frac12(d+1)(d+2) = 10$ from Lemma~\ref{E0-est} is
quite sharp since $E_{0,3} = \pi^2 > 9.86$.
\end{remark}

\begin{lemma}\label{rho-lemma}
There exists $0 \le \rho \in C^2(\R^d)$ such that $\spt\rho \subseteq B(0,1)$,\, $\int\rho = 1$ and
\begin{equation}\label{rho-est}
  \|\nabla\rho\|_1 \le d+1, \qquad \|\Delta\rho\|_1 \le 2(d+1)^2.
\end{equation}
\end{lemma}
\begin{proof}
Let $\rho_0 \in W_1^1(\R^d)$,\, $\rho_0(x) := \frac{d(d+2)}{2\sigma_{d-1}} (1-|x|^2)
\ind_{B(0,1)}(x)$, where $\sigma_{d-1}$ denotes the surface measure of the unit sphere
$\partial B(0,1)$. Then one easily computes
\[
  \int\rho_0 = 1, \qquad \|\nabla\rho_0\|_1 = \frac{d(d+2)}{d+1} < d+1
\]\smallbds
and
\[
  \Delta\rho_0 = \frac{d(d+2)}{\sigma_{d-1}}\bigl(-d\+\ind_{B(0,1)} + \delta_{\partial B(0,1)}\bigr)
\]
in the distributional sense, so $\Delta\rho_0$ is a measure with $\|\Delta\rho_0\| = 2d(d+2) <
2(d+1)^2$. Using a suitable mollifier and scaling, one obtains $\rho$ as asserted.
\end{proof}

\begin{proof}[Proof of Theorem~\ref{l1-est}]
(i) Let $r > \bigl(\frac{E_{0,d}}{t-\lambda}\bigr)^{1/2}$, and let $F_r := F_r(t)$,\, $G_r := G_r(t)$ be
as in~\eqref{FrGr}. Then $E_0(H_{G_r}) > \lambda$ by Lemma~\ref{E0-est}. We define $\xi\in
C^2(\R^d)$ satisfying
\[
  \spt \xi \subseteq G_r, \qquad \spt (\ind_{\R^d}-\xi) \subseteq U_{2r}(F_r)
\]
as follows: let $\rho_r := r^{-d} \rho(\rfrac\cdot r)$, where $\rho$ is as in
Lemma~\ref{rho-lemma}. Then
\[
  \xi := \ind_{\R^d} - \rho_{r/2} * \ind_{U_{3r/2}(F_r)}
     = \tfrac12 \ind_{\R^d} + \rho_{r/2} * (\tfrac12 \ind_{\R^d}-\ind_{U_{3r/2}(F_r)})
\]
has the above properties, and
\begin{equation}\label{xi-est}
\|\nabla\xi\|_\infty \le \tfrac12 \|\nabla\rho_{r/2}\|_1 = \rfrac1r \|\nabla\rho\|_1, \qquad
\|\Delta\xi\|_\infty \le \tfrac12 \|\Delta\rho_{r/2}\|_1 = \tfrac{2}{r^2} \|\Delta\rho\|_1.
\end{equation}
By Lemma~\ref{HG-action} we obtain $\xi\varphi \in D(H_{G_r})$ and
\[
  f_r := (H_{G_r}-\lambda)(\xi\varphi) = -2\nabla\xi\cdot\nabla\varphi - (\Delta\xi)\varphi, \qquad
  \spt f_r \subseteq \spt \nabla\xi \subseteq U_{2r}(F_r).
\]
Then $\xi\varphi = (H_{G_r}-\lambda)^{-1}f_r = (H_{G_r}-\lambda)^{-1}\ind_{U_{2r}(F_r)}f_r$. Since $\xi=1$ on $\Omega\setminus
U_{3r}(F_r)$, we can now estimate
\begin{equation}\label{phi-est}
\begin{split}
\|\varphi\|_1 &= \|\ind_{U_{3r}(F_r)}\varphi\|_1 + \|\ind_{\Omega\setminus U_{3r}(F_r)}\xi\varphi\|_1 \\
     &\le \|\ind_{U_{3r}(F_r)}\varphi\|_1
        + \| \ind_{\Omega\setminus U_{3r}(F_r)} (H_{G_r}-\lambda)^{-1} \ind_{U_{2r}(F_r)} \|_{1\to1} \|f_r\|_1.
\end{split}
\end{equation}
The remainder of the proof consists of estimating the terms in this pivotal
inequality.

Lemma~\ref{UsFr} implies
\begin{equation}\label{U3r}
  \|\ind_{U_{3r}(F_r)}\varphi\|_1 \le |U_{3r}(F_r)|^{1/2} \|\varphi\|_2 \le \bigl(\omega_d(5r)^dN_t(H)\bigr)^{1/2} \|\varphi\|_2
\end{equation}\smallbds
and
\begingroup\avoidbreak
\begin{align}
\|f_r\|_1 &\le |U_{2r}(F_r)|^{1/2} \|f_r\|_2 \le \bigl(\omega_d(4r)^dN_t(H)\bigr)^{1/2} \|f_r\|_2\,, \label{fr1} \\
\|f_r\|_2 &\le 2\|\nabla\xi\|_\infty\|\nabla\varphi\|_2 + \|\Delta\xi\|_\infty\|\varphi\|_2 \le \frac2r \+ \|\nabla\rho\|_1 \sqrt\lambda\,\|\varphi\|_2 + \frac{2}{r^2} \+
\|\Delta\rho\|_1\|\varphi\|_2\,, \label{fr2}
\end{align}
\endgroup
where in~\eqref{fr2} we used~\eqref{xi-est} and $\|\nabla\varphi\|_2^2 = \lambda\|\varphi\|_2^2$.

(ii) Next we estimate $\| \ind_{\Omega\setminus U_{3r}(F_r)} (H_{G_r}-\lambda)^{-1} \ind_{U_{2r}(F_r)} \|_{1\to1}$.
Let $\delta,\theta\in(0,1)$ and
\[
  \eps := \delta\+\frac{t-\lambda}{t}, \qquad \mu := (1-\eps)E_0(H_{G_r}) - \lambda.
\]
Then $(H_{G_r}-\lambda)^{-1} \le \eps^{-d/2} (\mu-\Delta)^{-1}$ by Theorem~\ref{res-est}, and hence
Proposition~\ref{free-res} implies
\begin{equation}\label{we1}
  \| \ind_{\Omega\setminus U_{3r}(F_r)} (H_{G_r}-\lambda)^{-1} \ind_{U_{2r}(F_r)} \|_{1\to1}
  \le \eps^{-d/2} (1-\theta^2)^{-d/2} \rfrac1\mu e^{-\theta r\sqrt\mu}.
\end{equation}

By Lemma~\ref{E0-est} and the definition of $\eps$ we have
\[
  \mu \ge (1-\eps)(t-E_{0,d}/r^2) - \lambda \ge t - \eps t - E_{0,d}/r^2 - \lambda
    = (1-\delta)(t-\lambda) - E_{0,d}/r^2.
\]
We now choose $r$ such that $r^2 = \frac{c^2+E_{0,d}}{(1-\delta)(t-\lambda)}$, with $c \ge
d+1$ to be determined later. Then
\begin{equation}\label{r-est}
  r^2 \le \frac{7/4}{(1-\delta)(t-\lambda)}\,c^2
\end{equation}
since $E_{0,d} \le \frac34(d+1)^2 \le \frac34 c^2$ by Lemma~\ref{E0-est}, and
\[
  \mu r^2 \ge (1-\delta)(t-\lambda)r^2 - E_{0,d} = c^2.
\]
By~\eqref{we1} we thus obtain
\begin{equation}\label{we2}
  \| \ind_{\Omega\setminus U_{3r}(F_r)} (H_{G_r}-\lambda)^{-1} \ind_{U_{2r}(F_r)} \|_{1\to1}
  \le \eps^{-d/2} (1-\theta^2)^{-d/2} \+ \smash[t]{\frac{r^2}{c^2}} \+ e^{-\theta c}.
\end{equation}

(iii) In this step we incorporate an estimate for $\|f_r\|_1$ into~\eqref{we2}.
By~\eqref{rho-est} we have $\|\nabla\rho\|_1 \le c$ and $\|\Delta\rho\|_1 \le 2c^2$. Thus,
using~\eqref{fr2}, \eqref{r-est} and $\lambda<t$ we obtain
\begin{align*}
\frac{r^2}{2} \|f_r\|_2
 &\le \|\nabla\rho\|_1 r\sqrt \lambda\,\|\varphi\|_2 + \|\Delta\rho\|_1\|\varphi\|_2 \\
 &\le c^2 \sqrt{\frac{7/4}{1-\delta}} \,\cdot \sqrt{\frac{\lambda}{t-\lambda}} \; \|\varphi\|_2 + 2c^2 \|\varphi\|_2
  \le c^2 C_\delta \sqrt{\frac{t}{t-\lambda}} \;\|\varphi\|_2\,,
\end{align*}
with $C_\delta = \sqrt{2/(1-\delta)}\++2$. Recalling $\eps = \delta\frac{t-\lambda}{t}$, we infer
by~\eqref{we2} that
\begin{equation}\label{we3}
\begin{split}
\| \ind_{\Omega\setminus U_{3r}(F_r)} (H_{G_r}-\lambda)^{-1} \ind_{U_{2r}(F_r)} \|_{1\to1} \|f_r\|_2
  \le \eps^{-d/2} & (1-\theta^2)^{-d/2} e^{-\theta c} \frac{r^2}{c^2} \|f_r\|_2 \\
  \le \delta^{-d/2} \bigl(\tfrac{t}{t-\lambda}\bigr)^{(d+1)/2} & (1-\theta^2)^{-d/2} e^{-\theta c} \cdot 2 C_\delta \|\varphi\|_2\,.
\end{split}
\end{equation}

Now we set $K_{\delta,\theta} := \frac54\delta(1-\theta^2)$ and choose
\[
  c := \frac{d+1}{2\theta} \ln\left( \frac{1}{K_{\delta,\theta}} \cdot \frac{t}{t-\lambda} \right).
\]
Then
\[
  \delta^{-d/2} \bigl(\tfrac{t}{t-\lambda}\bigr)^{(d+1)/2} (1-\theta^2)^{-d/2} e^{-\theta c}
  = \bigl(\tfrac54\bigr)^{d/2} K_{\delta,\theta}^{1/2},
\]
so by~\eqref{fr1} and~\eqref{we3} we obtain
\begin{equation}\label{we4}
\begin{split}
\MoveEqLeft[8] \| \ind_{\Omega\setminus U_{3r}(F_r)} (H_{G_r}-\lambda)^{-1} \ind_{U_{2r}(F_r)} \|_{1\to1} \|f_r\|_1 \\
 &\le \bigl(\omega_d(5r)^dN_t(H)\bigr)^{1/2} \cdot K_{\delta,\theta}^{1/2} \cdot 2C_\delta \|\varphi\|_2\,.
\end{split}
\end{equation}

(iv) We set $\theta:=\frac12$ and $\delta:=\frac{16}{45}$, so that $K_{\delta,\theta} = \frac13$ and
hence
\begin{equation}\label{c}
  c = (d+1) \ln\tfrac{3t}{t-\lambda} \ge d+1
\end{equation}
as required above. Moreover, one easily verifies that $\smash{K_{\delta,\theta}^{1/2}} \cdot
2C_\delta \le \frac92$. By~\eqref{phi-est}, \eqref{U3r} and~\eqref{we4} we conclude that
\begin{equation}\label{almost}
  \|\varphi\|_1^2 \le \Bigl( \tfrac{11}{2} \bigl(\omega_d(5r)^dN_t(H)\bigr)^{1/2} \|\varphi\|_2 \Bigr)^2
        = \bigl( \tfrac{11}{2} \bigr)^{\mkern-2mu 2} \omega_d (5r)^d N_t(H) \|\varphi\|_2^2\,.
\end{equation}

Stirling's formula yields
\[
  \omega_d = \frac{\pi^{d/2}}{\Gamma(\frac d2+1)}
     \le \frac{\pi^{d/2}}{\sqrt{2\pi d/2}\,\bigl(\frac{d}{2e}\bigr)^{d/2}}
     = (\pi d)^{-1/2} (2\pi e)^{d/2} d^{-d/2},
\]
so by~\eqref{r-est} we obtain
\[
  \omega_d (5r)^d
  \le (\pi d)^{-1/2} \bigl(2\pi e \cdot \tfrac{7/4}{1-\delta}\bigr)^{d/2} d^{-d/2} \cdot 5^d c^d(t-\lambda)^{-d/2}.
\]
Using $2\pi e \cdot \tfrac{7/4}{1-\delta} \le 7^2$,\, $(d+1)^d \le 2d^{d+1/2}$ and~\eqref{c} we
finally derive
\[
  \omega_d (5r)^d \le \pi^{-1/2} (7\cdot5)^d \cdot 2d^{d/2} \bigl(\ln\tfrac{3t}{t-\lambda}\bigr)^d (t-\lambda)^{-d/2}.
\]
Together with~\eqref{almost} this proves the assertion since
$\bigl( \frac{11}{2} \bigr)^{\mkern-2mu 2} \+ \pi^{-1/2}\cdot2 \le 35$.
\end{proof}

\end{document}